\documentclass[11pt,a4paper]{article}
\usepackage[top=1in,bottom=1.1in,left=1in,right=1in]{geometry}
\usepackage{amssymb}
\usepackage{amsmath,color}
\usepackage{amsthm}
\usepackage{hyperref}

\newtheorem{theorem}{Theorem}[section]
\newtheorem{proposition}[theorem]{Proposition}
\newtheorem{corollary}[theorem]{Corollary}
\newtheorem{example}[theorem]{Example}
\newtheorem{lemma}[theorem]{Lemma}
\theoremstyle{definition}
\newtheorem{definition}[theorem]{Definition}
\newtheorem{remark}[theorem]{Remark}

\DeclareMathOperator*{\argmin}{argmin}
\DeclareMathOperator*{\dom}{dom}
\DeclareMathOperator*{\epi}{epi}

\newcommand{\N}{\mathbb{N}}                                            % Symbol of natural numbers
\newcommand{\R}{\mathbb{R}}                                         % Symbol of real numbers                                  
\newcommand{\eps}{\varepsilon} 
 
\DeclareMathOperator{\CAT}{CAT}                                        % CAT(k) space.
\DeclareMathOperator{\Img}{Im}  
\newcommand{\inte}{\operatorname{int}}
\newcommand{\bd}{\operatorname{bd}}
\newcommand{\dist}{\operatorname{dist}}
\newcommand{\diam}{\operatorname{diam}}

\begin{document}

\title{Basic convex analysis in metric spaces with bounded curvature}

\author{Adrian S. Lewis$^{a}$, Genaro L\'{o}pez-Acedo$^{b}$, Adriana Nicolae$^{c}$}
\date{}
\maketitle

\begin{center}
{\scriptsize
$^{a}$School of Operations Research and Information Engineering, Cornell University, Ithaca, NY
\ \\
$^{b}$Department of Mathematical Analysis - IMUS, University of Seville, C/ Tarfia s/n, 41012 Seville, Spain
\ \\
$^{c}$Department of Mathematics, Babe\c s-Bolyai University, Kog\u alniceanu 1, 400084 Cluj-Napoca, Romania \\
\ \\
E-mail addresses: adrian.lewis@cornell.edu (A. S. Lewis), glopez@us.es (G. L\'{o}pez-Acedo),\\ 
anicolae@math.ubbcluj.ro (A. Nicolae)
}
\end{center}

\maketitle

\begin{abstract}
Differentiable structure ensures that many of the basics of classical convex analysis extend naturally from Euclidean space to Riemannian manifolds.  Without such structure, however, extensions are more challenging.  Nonetheless, in Alexandrov spaces with curvature bounded above (but possibly positive), we develop several basic building blocks. We define subgradients via projection and the normal cone, prove their existence, and relate them to the classical affine minorant property. Then, in what amounts to a simple calculus or duality result, we develop a necessary optimality condition for minimizing the sum of two convex functions. \\

\\

\end{abstract}

\section{Introduction}\label{intro}

 Extensions of convex analysis tools to Alexandrov spaces find motivating applications, e.g., in averaging phylogenetic trees \cite{Bac14} (in the tree space of Billera, Holmes, and Vogtmann \cite{BilHolVog01}, which is an Alexandrov space with nonpositive curvature but not a manifold), restoration of manifold-valued images \cite{BerPerSte16}, or a conjecture of Donaldson on the convergence of long time solutions of the Calabi flow in K\"{a}hler geometry \cite{Str16}. We also refer to Ba\v{c}\'{a}k's survey \cite{Bac18} for other interesting developments in optimization and analysis in Alexandrov spaces. Zhang and Sra \cite{ZhaSra16} point out the importance of optimization in nonlinear spaces and discuss applications (especially when the space is a Riemannian manifold) in machine learning and theoretical computer science. 

The subdifferential is the main analytic tool used to deal with nonsmooth convex functions on Euclidean space.  It is an intuitive notion, fulfilling the role of the derivative in first order optimality conditions for such functions \cite{Roc70, BC17}. Beyond the traditional setting of linear spaces, however, the question of how to define the subdifferential is less immediate.  Rather than relying on the subgradient inequality, the approach we take here follows one standard route in nonconvex variational analysis (see \cite{RocWet98, ClLeStWo98, Mor06}), deriving the subdifferential instead from the notion of the normal cone, an idea we can build through the metric projection.

In the nonlinear setting, the interest to define a suitable notion of subdifferential has found motivation in two facts: on the one hand, the study of the notion of gradient flow in spaces which are not necessarily endowed with a natural linear or differentiable structure (see \cite{May98,AmbGigSav05,GigNob21}) and, on the other hand, the analysis of convergence of some algorithms in Riemannian manifolds (see \cite{Smi94,LiLoMa09}).

As far as we know, the first notion of subdifferential was given in the nonlinear case by Udri\c{s}te \cite{Udri94} in the setting of Riemannian manifolds using the subgradient inequality. Important properties of the subdifferential such as its connection to the normal cone and the directional derivative were later extensively analyzed (see, e.g., \cite{LiLoMaWa11,LiMoWaYa11}) and, as a consequence, generalizations of classical first order algorithms have been developed in this framework too. It is worth to mention that, in the case of Riemannian manifolds, the linear structure of the tangent space allows mimicking  most of the constructions related to the subdifferential as well as the applications from the Euclidean case.

The case of spaces without a differentiable structure presents additional difficulties, mainly because of the lack of a fully linear structure in the tangent space. In  \cite{AhAm} and \cite{MoBhHo15}, a notion of subdifferential was proposed by first giving different definitions for the dual space and using in the definition of the subgradient inequality instead of the scalar product a quasi-linearization function in terms of distances introduced in \cite{BerNik08}. In both cases, the lack of a suitable structure of the dual space and the definition of the scalar product limit the study to the case of nonpositive curvature and essential properties of the subdifferential are left out, such as the existence of subgradients at a continuity point of a convex function or a subdifferential calculus.

In order to elaborate a more coherent theory, in the present work we introduce the concept of subdifferential using normal cones. Since the notion of normal cone can be based on the metric projection whose properties are rich enough in Alexandrov spaces of curvature bounded above (see \cite{Ari16}),  we consider this context and briefly discuss in Section~\ref{sect-prelim} some of its fundamental properties, together with other notions used in what follows. However, we also impose local compactness in our framework in order to obtain an analogue of the supporting hyperplane theorem from finite-dimensional Hilbert spaces, which allows us to establish the existence of subgradients at a continuity point of a convex function. We use the notion of tangent space and scalar product essentially
due to Berestovski\u{\i} (see \cite{AleBerNik86}) and discussed in detail in \cite{Bri99}.  With these notions, Gigli and Nobili \cite{GigNob21} introduced recently the concept of minus-subdifferential in connection to the study of gradient flows in Alexandrov spaces of curvature bounded above. This definition and the one we develop here are different (in Remark \ref{rmk-notion-GigNob}, we discuss the relation between these two definitions). Moreover, a characterization of the minus-subdifferential in terms of normal cones is not obvious. Another recent paper \cite{ChaKohKum21} uses the notion of tangent space to study monotone vector fields in Alexandrov spaces of nonpositive curvature mentioning as a particular example the subdifferential.

In Sections \ref{sect-normal-cone} and \ref{sect-subdiff} we introduce the notions of normal cone to a convex set and subdifferential of a convex function as elements of the tangent space and study their basic properties, along with some immediate examples. As pointed out in Section \ref{sect-tangent-sp}, in the case of smooth Riemannian manifolds, there is a natural identification between the space of tangent vectors at a point and the tangent space as considered here. Consequently, using  Remark \ref{rmk-normal-cone}(v)  for the normal cone  and Proposition \ref{prop-subdiff-convex} for the subdifferential, our definitions coincide with the ones introduced in \cite{Udri94} and \cite{LiMoWaYa11} in this setting. 

To convey our culminating result -- Theorem \ref{thm-sum-rule}, which may be seen as a basic calculus or duality theorem -- it helps to review the classical case in a Euclidean space  $X$.  Given two convex functions  $f, g : X \to (-\infty, \infty]$, an archetypal decomposition question of a kind ubiquitous in modern optimization seeks to characterize points $x$ minimizing the sum $f+g$. A sufficient condition is trivial: if $y$ is a subgradient of $f$ at $x$,  meaning classically that $x$ minimizes the function  $ z \in X \mapsto f(z) - \langle y, z\rangle$,  and the opposite vector $-y$ is a  subgradient of  $g$ at $x$, then $x$ minimizes $f+g$.  Expressed concisely, the existence of opposite subgradients at $x$ is a sufficient condition for $x$ to minimize the sum.  More generally, the sum of any subgradient of $f$ and of $g$ at $x$ is easily seen to be a subgradient at $x$ of $f+g$.

Less obvious is the converse question, a necessary condition.  If a point $x$ minimizes the sum $f+g$, must there exist two opposite subgradients, $y$ and $-y$, for the two functions? A counterexample is the sum of the function $f : \R \to (-\infty,\infty]$ defined  by 
\[f(x) = \begin{cases} \infty, & \mbox{if } x < 0,\\ -\sqrt{x}, & \mbox{if } x \ge 0,\\ \end{cases}\]
with the indicator function of the set of nonpositive reals $\delta_{\R_{-}} : \R \to [0,\infty]$ defined by  
\[\delta_{\R_{-}}(x) = \begin{cases} 0, & \mbox{if } x \le 0,\\ \infty, & \mbox{if } x > 0.\\ \end{cases}\]
However, if $f$ is continuous at any point where $g$  is finite, then the answer is yes, and indeed, more generally, subgradients of $f+g$ at any point are characterized in that case as sums of subgradients of $f$ and of $g$ there.  Our culminating result, Theorem \ref{thm-sum-rule}, is a version of this result extended to Alexandrov spaces.

The existence of the opposite subgradients, $y$ and $-y$,  is the most basic form of the Fenchel duality theorem, and many modern optimization algorithms for minimizing the primal objective $f(x) + g(x)$ seek in tandem such a vector $y$ by implicitly maximizing a dual objective $-f^*(y) - g^*(-y)$ involving the conjugates  $f^*(y) = \sup_{z \in X}\left(\langle y, z \rangle - f(z)\right)$  of  $f$  and  $g^*$  of  $g$.  A popular general-purpose example is the alternating directions method of multipliers, surveyed in \cite{BoyParChu11}. For elementary reasons, the dual objective is never larger than the primal, and equality guarantees that the corresponding vectors $x$ and $y$ are primal and dual optimal, a case that holds exactly when $y$ and $-y$  are the requisite opposite subgradients of $f$ and $g$ at $x$. Whether Theorem \ref{thm-sum-rule} has an analogous dual interpretation we do not pursue here. However, in Corollary \ref{cor-separation}, we give an application to a counterpart in Alexandrov spaces of a separation result for two nonempty, convex, closed, and disjoint sets, at least one of which is bounded. As an optimization tool, separation of disjoint convex sets needs no emphasis, either for fundamentals \cite{Roc70, MorNam22} or algorithms \cite{Ber15}, whence our broad interest in nonlinear extensions such as Corollary \ref{cor-separation}.  A particular and fundamental machine learning example is the ``support vector machine'', which, in its simplest form, trained on some binary-labeled data points in Euclidean space, classifies new points using a ``support vector'' determining a separating hyperplane for the two convex hulls \cite{SraNowWri11}. If the data instead lie in a manifold (of low rank matrices, for example) or some more general nonlinear space, then rather than resorting to the lifting procedures standard for support vector machines, we speculate that the pair of opposite tangent elements in our nonlinear separation result, Corollary \ref{cor-separation}, could serve as the classifier. This is a topic of ongoing investigation.

\section{Preliminaries}\label{sect-prelim}

In this section, we give some basic notions and properties of geodesic metric spaces. We refer the reader to \cite{AleKapPet23,Bri99,Bur01} for more details.

\subsection{Geodesic metric spaces and $\CAT(\kappa)$ spaces}

Let $(X,d)$ be a metric space. We denote the open (resp., closed) ball centered at $x\in X$ with radius $r>0$ by $B(x,r)$ (resp., $\overline{B}(x,r)$). For $C \subseteq X$, we denote the {\it diameter} of $C$ by $\diam(C)$. The {\it metric projection} $P_C$ onto $C$ is the mapping $P_C : X\to 2^C$ defined by
\[P_C(z)=\{ y \in C : d(z,y)=\dist(z,C)\} \quad \text{for all } z\in X,\]
where  $\dist(z,C) = \inf_{y \in C}d(z,y)$.

A {\it geodesic} is an isometric mapping $\gamma:[0,l] \subseteq \R \to X$. The image $\gamma([0,l])$ is called a {\it geodesic segment}. If instead of the interval $[0,l]$ one considers $\R$, then the image of $\gamma$ is called a {\it geodesic line}. If every two points in $X$ are joined by a (unique) geodesic segment, then $X$ is called a {\it (uniquely) geodesic space}. We say that a subset $C$ of a geodesic space is {\it convex} if for every $x,y \in C$, all geodesic segments with endpoints $x$ and $y$ are contained in $C$. 

Assume that $(X,d)$ is a uniquely geodesic space. For $x,y \in X$, denote the unique geodesic segment with endpoints $x$ and $y$ by $[x,y]$ and, given $t \in [0,1]$, let $(1-t)x+ty$ stand for the unique point belonging to $[x,y]$ whose distance to $x$ equals $td(x,y)$. A set $C \subseteq X$ is said to satisfy the {\it betweenness property} if for every four pairwise distinct points $x, y, z, w \in C$, if $y \in [x,z]$ and $z \in [y,w]$, then $y,z \in [x,w]$. This property was studied in \cite{DimWhi81, Nic13, KohLopNic21}.

We say that $X$ has the {\it geodesic extension property around $x \in X$} if there exists a positive constant $R$ such that for any distinct $y,z \in B(x,R)$ with $d(y,z) < R$, the geodesic from $y$ to $z$ can be extended beyond $z$ to a geodesic of length $R$. We will often emphasize the dependence of the constant on the point by denoting it $R_x$. We say that $X$ has the {\it geodesic extension property} if every geodesic segment (that is not reduced to a point) is contained in a geodesic line. One can show that if $X$ is complete and satisfies the betweenness property and the geodesic extension property around every $x \in X$, then $X$ has the geodesic extension property.

Let $(X,d)$ be a metric space, and consider on the Cartesian product $X \times \R$ the metric
\[d_2((x_1,y_1),(x_2,y_2)) = \sqrt{d(x_1,x_2)^2 + |y_1 - y_2|^2},\]
where $x_1, x_2 \in X$ and $y_1, y_2 \in \R$. 

If $X$ is locally compact, then $X \times \R$ is locally compact as well. At the same time, if $X$ is a (uniquely) geodesic space, then so is $X \times \R$. Moreover, geodesics in $X \times \R$ are given in terms of geodesics in $X$ and in $\R$: if $\gamma : [0,r] \to X$ and $\alpha : [0,s] \to \R$ are geodesics, where $l = \sqrt{r^2 + s^2} > 0$, then $\sigma : [0,l] \to X \times \R$ defined by $\sigma(t) = (\gamma(rt/l), \alpha(st/l))$ for all $t \in [0,l]$ is a geodesic. Conversely, if $l > 0$ and $\sigma : [0,l] \to X \times \R$ is a geodesic from $(x,\lambda)$ to $(y,\omega)$, taking $r = d(x,y)$ and $s = |\lambda - \omega|$, then $\sigma(t) = (\gamma(rt/l), \alpha(st/l))$ for all $t \in [0,l]$, where $\gamma : [0,r] \to X$ and $\alpha : [0,s] \to \R$ are geodesics from $x$ to $y$ and from $\lambda$ to $\omega$, respectively.

Suppose that $(X,d)$ is a geodesic space. Given $\lambda \in \R$, if $X$ has the geodesic extension property around $x \in X$, then $X \times \R$ also has the geodesic extension property around $(x,\lambda)$ with the same constant.

A function $f : X \to (-\infty,\infty]$ is called {\it convex} if for any geodesic $\gamma : [0,l] \to X$ and any $t \in [0,1]$,
$f(\gamma(lt)) \le (1-t)f(\gamma(0)) + t f(\gamma(l))$. The (effective) {\it domain} of $f$ is defined by $\dom f = \{x \in X \mid f(x) < \infty\}$. The {\it epigraph} of $f$ is $\epi f = \{ (x,\lambda) \in \dom f \times \R \mid \lambda \ge f(x)\}$. One can easily see that $\epi f$ is closed if and only if $f$ is lower semicontinuous, which is also equivalent to the fact that its sublevel sets are closed. If $f$ is convex, then $\dom f$ and $\epi f$ are convex.

For $\kappa \in \R$, let $M_\kappa^2$ denote the complete, simply connected, $2$-dimensional Riemannian manifold of constant sectional curvature $\kappa$. We denote the diameter of $M_\kappa^2$ by $D_\kappa$. In other words, $D_\kappa = \infty$ if $\kappa \le 0$, while $D_\kappa = \pi/\sqrt{\kappa}$ if $\kappa > 0$. 

A {\it geodesic triangle} is the union of three geodesic segments joining three points. We say that a triangle $\Delta(\overline{x}_1, \overline{x}_2, \overline{x}_3)$ in $M_\kappa^2$ is a {\it comparison triangle} for a geodesic triangle $\Delta(x_1,x_2,x_3)$ if $d(x_i,x_j) = d_{M_\kappa^2}(\overline{x}_i,\overline{x}_j)$ for all $i,j \in \{1,2,3\}$. 

A metric space is called a {\it $\CAT(\kappa)$ space} (also known as a space of curvature bounded above by $\kappa$ in the sense of Alexandrov) if every two points at distance less than $D_\kappa$ can be joined by a geodesic and geodesic triangles having perimeter less than $2D_\kappa$ are not thicker than the comparison triangles in $M_\kappa^2$.

Suppose next that $X$ is a $\CAT(\kappa)$ space. Points in $X$ at distance less than $D_\kappa$ are joined by a unique geodesic segment and this segment varies continuously with its endpoints. Moreover, balls of radius smaller than $D_\kappa/2$ are convex. If $\kappa \ge 0$, $X \times \R$ is a $\CAT(\kappa)$ space as well. 

Sets of diameter less than $D_\kappa$ if $\kappa > 0$ satisfy the betweenness property. This shows that if a complete $\CAT(0)$ space has the geodesic extension property around every point, then it also has the geodesic extension property.  

If the diameter of $X$ is smaller than $D_\kappa/2$ when $\kappa > 0$, then $X$ is {\it $2$-uniformly convex} (see \cite{Kuw14, Oht07}) in the sense that there exists a parameter $K = K(\kappa, \diam(X)) >0$ such that for all $x,y,z\in X$ and all $t\in[0,1]$,
\[d(z, (1-t)x + ty)^2\leq(1-t)\,d(z,x)^2+t\,d(z,y)^2-K\,t(1-t)\,d(x,y)^2.\]
The above inequality with $K = 1$ characterizes $\CAT(0)$ spaces, in fact.

The Alexandrov angle between two nonconstant geodesics $\gamma:[0,l]\to X$ and $\gamma':[0,l']\to X$ issuing at the same point $x = \gamma(0) = \gamma'(0)$ is well-defined and can be determined as $\angle(\gamma,\gamma') = \lim_{t,t' \searrow 0} \angle_{\overline{x}}\left(\overline{\gamma(t)},\overline{\gamma'(t')}\right)$, 
where $\Delta(\overline{\gamma(t)}, \overline{x}, \overline{\gamma'(t')})$ is a comparison triangle in $\R^2$. We also denote it by $\angle_x(y,z)$, where $y \in \gamma((0,l])$ and $z \in \gamma'((0,l'])$.

Let $C \subseteq X$ be a convex set and $x \in X \setminus C$ with $\dist(x,C) < D_\kappa/2$. If $P_C(x) \ne \emptyset$, then $P_C(x)$ is a singleton. If $z = P_C(x)$ and $y \in C$ with $y \ne z$, then $\angle_z(x,y) \ge \pi/2$. If $C$ is additionally complete in the induced metric, then $x$ always has a nearest point in $C$, and hence $P_C(x)$ is a singleton. We refer the reader to \cite{EspFer09, Ari16} for a more thorough discussion on the behavior of the metric projection in $\CAT(\kappa)$ spaces.

\subsection{Tangent spaces}\label{sect-tangent-sp}

{\it In what follows, if nothing else is mentioned about the context, we always suppose that $(X,d)$ is a locally compact $\CAT(\kappa)$ space, where $\kappa > 0$. To simplify the exposition and without loss of generality (see Remark \ref{rmk-normal-cone}(iv)), we also suppose that $X$ is uniquely geodesic. In addition, assume that for all $x \in X$, $X$ has the geodesic extension property around $x$ with a constant smaller than $D_\kappa/2$. This constant will be denoted by $R_x$.} (Actually, in the subsequent discussion we only need to impose the geodesic extension property and the existence of a compact neighborhood around a fixed point where we consider the tangent space of $X$.)

If $x$ is a point in a smooth Riemannian manifold, then there exists a neighborhood of $x$ that is a $\CAT(\kappa)$ space for some suitable $\kappa \in \R$ (see, e.g., \cite[Chapter II.1, Appendix]{Bri99}). Moreover, since the injectivity radius in a Riemannian manifold is a continuous function (see \cite{Bou23}), there exists a ball centered at $x$ that is a $\CAT(\kappa)$ space having the geodesic extension property around $x$ with an appropriate constant.

Fix $x \in X$. Denote by $\Theta_x X$  the set of all nonconstant geodesics issuing at $x$. The Alexandrov angle induces a metric on the set $\Sigma_x X$ of equivalence classes of geodesics in $\Theta_x X$, where two geodesics $\gamma, \eta\in \Theta_x X$ are considered equivalent if $\angle(\gamma,\eta) = 0$. Note that if $X$ is a Riemannian manifold, then $\Sigma_x X$ is isometric to the unit sphere in the tangent space of $X$ at $x$.

For a geodesic $\gamma \in \Theta_x X$, we denote by $[\gamma]$ its equivalence class. At the same time, when working with an equivalence class in $\Sigma_x X$, we assume that it is represented by a geodesic in $\Theta_x X$. Since $X$ has the geodesic extension property around $x$, we can always suppose that a representative of an equivalence class is a geodesic of length at least $R_x$, where $R_x$ is the constant for the geodesic extension property around $x$. 

Under our assumptions, the metric space $(\Sigma_x X, \angle)$ is complete. We will refer to the elements in $\Sigma_x X$ as {\it directions} at $x$ and to $(\Sigma_x X,\angle)$ as the {\it space of directions} at $x$. The {\it tangent space} $T_x X$ of $X$ at $x$ is the Euclidean cone over the metric space $(\Sigma_x X, \angle)$ (see \cite[Chapter I, Definition 5.6]{Bri99}). More precisely, $T_x X = (\Sigma_x X \times [0,\infty))/\sim$, where for $([\gamma],r), ([\eta],s) \in \Sigma_x X \times [0,\infty)$, $([\gamma],r) \sim ([\eta],s)$ if and only if ($r = s = 0$) or ($r = s > 0$ and $[\gamma] = [\eta]$). The equivalence class of $([\gamma],0) \in T_x X$ is called the {\it origin} of $T_x X$ and will be denoted by $o_x$. (When dealing with the origin $o_x$, the direction at $x$ bears no relevance and sometimes, to simplify the writing, we consider the geodesic constantly equal to $x$.)

Observe that despite the terminology, this construction is, in general, merely a metric cone. However, in the case when $X$ is a Riemannian manifold, this notion coincides with the traditional notion of tangent space. Besides \cite{Bri99, AleKapPet23}, the preliminary part of \cite{LanSch97} describes in detail the construction and properties of tangent spaces.

The {\it metric} $d_x$ on $T_x X$ is defined for $v = ([\gamma],r), w = ([\eta],s) \in T_x X$ by 
\[d_x(v,w)^2 = r^2 + s^2 - 2rs\cos \angle (\gamma,\eta),\]
while the {\it scalar product} of $v$ and $w$ is given by $\langle v,w \rangle = rs\cos \angle(\gamma, \eta)$. The multiplication with a nonnegative scalar $\lambda \ge 0$ is defined for $v =  ([\gamma],r)$ by $\lambda v = ([\gamma],\lambda r)$. We say that  $v = ([\gamma],r), w = ([\eta],s) \in T_x X$ are {\it opposite} to each other if ($r = s = 0$) or ($r = s > 0$ and $\angle(\gamma,\eta) = \pi$). Observe that, under our assumptions on $X$, $T_x X$ is a complete $\CAT(0)$ space (see \cite{Nik95,Bri99}). 

We finish this section with the following auxiliary result concerning the tangent space of $X \times \R$.

\begin{lemma}\label{lemma-add-scalar-prod}
Let $(x,\lambda) \in X \times \R$, and, for $i \in \{1,2\}$, let $\sigma_i : [0,l_i] \to X \times \R$ be a nonconstant geodesic issuing at $(x,\lambda)$ that is written as $\sigma_i(t) = \left(\gamma_i(r_it/l_i), \alpha_i(s_it/l_i)\right)$ for all $t \in [0,l_i]$, where $l_i^2 = r_i^2 + s_i^2$, and $\gamma_i : [0,r_i] \to X$ and $\alpha_i : [0,s_i] \to \R$ are geodesics. Then the following hold:
\begin{enumerate}
\item[(i)] $\langle ([\sigma_1],l_1), ([\sigma_2],l_2) \rangle = \langle  ([\gamma_1],r_1), ([\gamma_2],r_2)\rangle + \langle  ([\alpha_1],s_1), ([\alpha_2],s_2) \rangle$.
\item[(ii)] $[\sigma_1] = [\sigma_2]$ if and only if $\left([\gamma_1] = [\gamma_2], [\alpha_1] = [\alpha_2], r_1/l_1 = r_2/l_2, s_1/l_1=s_2/l_2\right)$. 
\end{enumerate}
\end{lemma}
\begin{proof}
Note first that for $i \in \{1,2\}$, $\alpha_i(t) = \lambda + t$ for all $t \in [0,s_i]$ or $\alpha_i(t) = \lambda - t$ for all $t \in [0,s_i]$.
If for both $i=1$ and $i=2$, $\alpha_i$ has the same form, let $c = 1$. Otherwise, let $c = -1$. Denote also $a_i = r_i/l_i$ and $b_i = s_i/l_i$, where $i \in \{1,2\}$.

(i) Take $t \in (0,\min\{l_1,l_2\}]$. For the geodesic triangle $\Delta(\sigma_1(t), (x,\lambda), \sigma_2(t))$, consider a comparison triangle $\Delta(\overline{\sigma_1(t)}, \overline{(x,\lambda)}, \overline{\sigma_2(t)})$ in $\R^2$ and denote by $\theta(t)$ its interior angle at $\overline{(x,\lambda)}$. Then $\cos\theta(t) = 1 - d_2(\sigma_1(t),\sigma_2(t))^2/(2t^2)$.
As 
\[d_2(\sigma_1(t),\sigma_2(t))^2 = d(\gamma_1(a_1t), \gamma_2(a_2t))^2 + (b_1t)^2 + (b_2t)^2 -2 c b_1b_2t^2,\] 
it follows that 
\[\cos \theta(t) = \frac{1}{2}\left(a_1^2 + a_2^2 - \frac{d(\gamma_1(a_1t), \gamma_2(a_2t))^2}{t^2}\right) +c b_1b_2.\]
Now take a comparison triangle  $\Delta(\overline{\gamma_1(a_1t)}, \overline{x}, \overline{\gamma_2(a_2t)})$ in $\R^2$ and denote by $\rho(t)$ its interior angle at $\overline{x}$. Then 
\[\cos \rho(t) = \frac{1}{2a_1a_2}\left(a_1^2 + a_2^2 - \frac{d(\gamma_1(a_1t), \gamma_2(a_2t))^2}{t^2}\right).\]
Hence, $\cos \theta(t) = a_1a_2\cos \rho(t) + cb_1b_2$. Taking limit as $t \searrow 0$, we conclude that 
\begin{equation}\label{lemma-add-scalar-prod-eq3}
\cos \angle(\sigma_1,\sigma_2) = a_1a_2\cos \angle(\gamma_1,\gamma_2) + cb_1b_2,
\end{equation}
from where one obtains the desired equality.

(ii) If $[\gamma_1] = [\gamma_2]$, $[\alpha_1] = [\alpha_2]$, $a_1 = a_2$, and $b_1 = b_2$, then $\cos \angle(\gamma_1,\gamma_2)=1$, $c=1$, and \eqref{lemma-add-scalar-prod-eq3} gives $\cos \angle(\sigma_1,\sigma_2) = a_1^2 + b_1^2 = 1$, and hence $[\sigma_1] = [\sigma_2]$.

Suppose now $[\sigma_1] = [\sigma_2]$. Then $\cos \angle(\sigma_1,\sigma_2) =1$ and, using again \eqref{lemma-add-scalar-prod-eq3}, we obtain that
$1 =  a_1a_2\cos \angle(\gamma_1,\gamma_2) + cb_1b_2 \le a_1a_2 + b_1b_2,$
so
\[1 \le (a_1a_2 + b_1b_2)^2  = (a_1^2+b_1^2)(a_2^2+b_2^2) - (a_1b_2 - a_2b_1)^2 = 1 - (a_1b_2 - a_2b_1)^2.\]
We conclude that $\cos \angle(\gamma_1,\gamma_2) = 1$, $c = 1$, and $a_1b_2 = a_2b_1$. Hence, $[\gamma_1] = [\gamma_2]$ and $[\alpha_1] = [\alpha_2]$. Again taking into account that $a_1^2 + b_1^2 = a_2^2 + b_2^2 = 1$, we obtain $a_1 = a_2$ and $b_1 = b_2$.
\end{proof}

\section{Normal cones}\label{sect-normal-cone} 

Traditionally, a central idea in smooth analysis is the approximation of a smooth function and a smooth manifold by a linear function and a linear subspace, respectively. In the nonsmooth case, a function is approximated by a family of linear functions, while a set is approximated by cones. Normal cones can be defined in terms of the metric projection. 
In $\R^n$, given a convex set $C$ and a point $x \in C$, vectors in
\[N_C(x) = \{v \in \R^n\mid x = P_C(x+v)\} = \{v \in \R^n\mid \langle y - x, v\rangle \le 0 \text{ for all } y \in C\}\]
are called normals to $C$ at $x$. 

The metric projection is in fact a purely metric notion and exhibits a sufficiently regular behavior in nonlinear settings with a rich enough geometry (see \cite{Ari16}). For this reason, it is natural to consider using normal cones to define the notion of subdifferential of a function and get a meaningful subdifferential calculus in nonlinear spaces. 

Let $C \subseteq X$ be nonempty and convex, and let $x \in C$. 

\begin{definition}\label{def-prox-normal-cone}
The {\it normal cone} to $C$ at $x$ is given by
\[N_C(x) = \{([\gamma],r) \in T_x X \mid r = 0 \text{ or } \left(x = P_C(\gamma(t)) \text{ for some } t > 0\right)\}.\]
Elements in $N_C(x)$ are called {\it normals} to $C$ at $x$.
\end{definition}

In connection to this definition, the following facts are immediate.
\begin{remark} \label{rmk-def-proximal-normal-cone}
\begin{enumerate}
\item[(i)] If $x = P_C(\gamma(t))$, then $x = P_C(\gamma(s))$ for all $s \in [0,R_x]$.

\item[(ii)] $N_C(x)$ is well-defined in the sense that if $\gamma, \eta \in \Theta_x X$ with $\eta \in [\gamma]$ and $x = P_C(\gamma(t))$ for some $t > 0$, then 
$x = P_C(\eta(s))$ for all $s \in [0,R_x]$. 

Suppose there exists $y \in C$ such that $d(y, \eta(R_x)) < d(x, \eta(R_x))$. Then
\[\pi/2 \le \angle_{x}(\gamma(t),y) \le  \angle_{x}(\gamma(t),\eta(R_x)) + \angle_{x}(\eta(R_x),y) = \angle_{x}(\eta(R_x),y),\]
from where $d(y, \eta(R_x)) \ge d(x, \eta(R_x))$, a contradiction. We conclude that $x = P_C(\eta(R_x))$, from where $x = P_C(\eta(s))$ for all $s \in [0,R_x]$.
\end{enumerate}
\end{remark}

In the next remark, we collect some elementary properties of normal cones.
\begin{remark} \label{rmk-normal-cone}
\begin{enumerate}
\item[(i)] $N_C(x)$ is closed under multiplication with a nonnegative scalar. Thus, for $r > 0$, $([\gamma],r) \in N_C(x)$ if and only if $([\gamma],1) \in N_C(x)$. 

\item[(ii)] If $\diam(C) < D_\kappa$, then $\overline{C}$ is convex and $N_C(x) = N_{\overline{C}}(x)$.

\item[(iii)] If $D \subseteq C$ is convex and $x \in D$, then $N_C(x) \subseteq N_D(x)$.

\item[(iv)] Let $V$ be a convex neighborhood of $x$. Then $N_C(x) = N_{C \cap V}(x)$. This allows us to assume that the diameter of the set $C$ is as small as needed. 

Indeed, let $r > 0$ such that $B(x,r) \subseteq V$. Using (iii), it is enough to show that if $([\gamma],1) \in N_{C \cap V}(x)$, then $([\gamma],1) \in N_C(x)$. Suppose $([\gamma],1) \in N_{C \cap V}(x)$. Then there exists $t \in (0,r/2)$ such that $x = P_{C \cap V}(\gamma(t)) $. For $y \in C \setminus V$, 
\[r \le d(x,y) \le d(x,\gamma(t)) + d(\gamma(t),y) < r/2 + d(\gamma(t),y),\] 
so $d(\gamma(t),y) > r/2 > d(\gamma(t),x)$. This shows that  $x = P_{C}(\gamma(t))$, and we conclude that $([\gamma],1) \in N_C(x)$.

\item[(v)] $N_C(x) = \{([\gamma],r) \in T_x X \mid \langle ([\gamma],r),([\gamma^y],d(x,y)) \rangle \le 0 \text{ for all } y \in C\},$ where for $y \in C$, $\gamma^y$ is the geodesic from $x$ to $y$. 

Indeed, $([\gamma],1) \in N_C(x)$ if and only if $x = P_C(\gamma(R_x))$ if and only if $ \angle(\gamma,\gamma^y) \ge \pi/2$ for all $y \in C$ with $y \ne x$.

\item[(vi)] Let $(x_n)$ be a sequence in $C$ converging to $x$, $l \in (0,D_\kappa/2)$, and, for $n \in \N$, let $\gamma_n : [0,l] \to X$ be a nonconstant geodesic issuing at $x_n$. If $([\gamma_n],1) \in N_C(x_n)$ for all $n \in \N$ and $(\gamma_n)$ converges pointwise to a geodesic $\gamma$, then $([\gamma],1) \in N_C(x)$.

Indeed, $P_C(\gamma_n(l)) = x_n$ for all $n \in \N$. Let $y \in C$. Then
\begin{align*}
d(\gamma(l),x) &\le d(\gamma(l), \gamma_n(l)) + d(\gamma_n(l),x_n) + d(x_n,x)\\
&  \le d(\gamma(l), \gamma_n(l)) + d(\gamma_n(l),y) + d(x_n,x),
\end{align*}
for all $n \in \N$. We get that $d(\gamma(l),x) \le d(\gamma(l),y)$ for all $y \in C$, which shows that $P_C(\gamma(l)) = x$, so $([\gamma],1) \in N_C(x)$.

\end{enumerate}
\end{remark}

We also have the following result.

\begin{proposition} \label{prop-normal-cone-int}
\begin{enumerate}
\item[(i)] If $x \in \inte(C)$, then $N_C(x) = \{o_x\}$.
\item[(ii)] If $N_C(x) = \{o_x\}$, then $x \in \inte(\overline{C})$. 
\end{enumerate}
\end{proposition}
\begin{proof}
(i) Let  $x \in \inte(C)$. If $([\gamma],1) \in N_C(x)$, then $x = P_C(\gamma(t))$ for some $t > 0$, which contradicts the fact that $x \in \inte(C)$.

(ii) Let $R \in (0,R_x]$ such that $\overline{B}(x,R)$ is compact, and take the set $D = C \cap \overline{B}(x,R)$. As $\{o_x\} = N_C(x) = N_D(x) = N_{\overline{D}}(x)$, we can assume that $D$ is closed, and hence compact (otherwise, consider $\overline{D}$ instead of $D$). Suppose that $x \notin \inte(\overline{C})$. Then $x \in D \setminus \inte(D)$. 

Take a sequence $(x_n) \subseteq X \setminus D$ with $x_n \to x$ and $d(x_n, x) \le R/2$ for all $n \in \N$. For $n \in \N$, take $p_n = P_{D}(x_n)$. Observe that $d(x, p_n) \le d(x,x_n) + d(x_n,p_n) \le 2d(x,x_n)$ for all $n \in \N$, and hence $p_n \to x$.

Extend the geodesic from $p_n$ to $x_n$ beyond $x_n$ to a point $x'_n$ so that $d(p_n,x'_n) = R/2$. Because $(x'_n) \subseteq \overline{B}(x,R)$, the sequence $(x_n')$ has a convergent subsequence whose limit we denote by $x'$. For $n \in \N$, denote by $\gamma_n$ the geodesic from $p_n$ to $x'_n$ and let $\gamma$ be the geodesic from $x$ to $x'$. Since $([\gamma_n],1) \in N_{D}(p_n)$ for all $n \in \N$, applying Remark \ref{rmk-normal-cone}(vi), we get $([\gamma],1) \in  N_D(x)$, a contradiction.
\end{proof}

The above result can be interpreted as an analogue of the supporting hyperplane theorem from finite-dimensional Hilbert spaces. Indeed, if $C$ is convex and closed and $x \in \bd(C)$, there exists $([\gamma],1) \in N_C(x)$, and hence, by Remark \ref{rmk-normal-cone}(v), $\langle ([\gamma],1),([\gamma^y],d(x,y)) \rangle \le 0$ for all $y \in C$.

We finish this section with the following straightforward examples.

\begin{example} \label{ex-normal-cone}
\item[(i)] $N_{\{x\}}(x) = T_x X$.
\item[(ii)] If $\gamma$ and $\eta$ are geodesics issuing at $x$, then $([\gamma],1) \in N_{\Img \eta}(x)$ if and only if $([\eta],1) \in N_{\Img \gamma}(x)$.
\item[(iii)] Let $z \in X$ and $r \in (0, D_\kappa/2)$. Suppose that $x \in \bd(\overline{B}(z, r))$, and denote by $\Gamma$ the set of all nonconstant geodesics $\gamma : [0,l] \to X$ issuing at $x$ with the property that $d(z,x) + l = d(z,\gamma(l))$. Then $([\gamma],1) \in N_{\overline{B}(z, r)}(x)$ if and only if $\gamma \in \Gamma$.
\item[(iv)] Given $x \in X$, in the geodesic space $X \times \R$ we have $([\sigma],1) \in N_{X \times \R_+}(x,0)$ if and only if $\sigma : [0,l] \to X \times \R$ is defined by $\sigma(t) = (x,-t)$ for all $t \in [0,l]$.
\end{example}
\begin{proof}
(iii) If $([\gamma],1) \in N_{\overline{B}(z, r)}(x)$, then $x = P_{\overline{B}(z, r)}(\gamma(R_x))$. Let $y \in [z, \gamma(R_x)]$ with $d(z,y) = r$. Then $y \in \overline{B}(z, r)$ and $d(\gamma(R_x),y) \ge d(\gamma(R_x),x)$, so
\[d(z,\gamma(R_x)) \le d(z,x) + d(x,\gamma(R_x)) \le d(z,y) + d(y,\gamma(R_x)) = d(z,\gamma(R_x)).\]
We conclude that $\gamma \in \Gamma$ and, by uniqueness of geodesics, $x=y$.

Conversely, if $\gamma \in \Gamma$, then for all $y \in \overline{B}(z, r)$,
\[d(z,x) + d(x,\gamma(R_x)) = d(z,\gamma(R_x)) \le d(z,y) + d(y,\gamma(R_x)),\]
from where $d(\gamma(R_x),x) \le d(\gamma(R_x),y)$. Thus, $x = P_{\overline{B}(z, r)}(\gamma(R_x))$, so $([\gamma],1) \in N_{\overline{B}(z, r)}(x)$.

(iv) Let $\sigma : [0,l] \to X \times \R$ be a nonconstant geodesic issuing at $(x,0)$ that is written as $\sigma(t) = \left(\gamma(rt/l), \alpha(st/l)\right)$ for all $t \in [0,l]$, where $l^2 = r^2 + s^2$, and $\gamma : [0,r] \to X$ and $\alpha : [0,s] \to \R$ are geodesics. Note that $\alpha(t) = t$ for all $t \in [0,s]$ or $\alpha(t) = -t$ for all $t \in [0,s]$.

If $([\sigma],1) \in N_{X \times \R_+}(x,0)$, then $P_{X \times \R_+}(\sigma(t)) = (x,0)$ for some $t \in (0,l]$. Thus,
\[d_2(\sigma(t),(x,0)) \le d_2(\sigma(t),(x,\lambda)),\]
so $|\alpha(st/l)| \le |\alpha(st/l) - \lambda|$ for all $\lambda \ge 0$ and we get that $\alpha(t) = -t$ for all $t \in [0,s]$. Moreover,
\[d_2(\sigma(t),(x,0)) \le d_2(\sigma(t),(\gamma(rt/l),0)),\]
from where $\gamma(rt/l) = x$ and $r=0$.

Conversely, if $\sigma : [0,l] \to X \times \R$ is defined by $\sigma(t) = (x,-t)$ for all $t \in [0,l]$, then 
\[d_2(\sigma(t),(x,0))^2 = t^2 \le d(x,y)^2 + (t+\lambda)^2 = d_2(\sigma(t),(y,\lambda))^2,\]
for all $t \in [0,l]$ and all $(y,\lambda) \in X \times \R_+$. This shows that $([\sigma],1) \in N_{X \times \R_+}(x,0)$.
\end{proof}

\section{The subdifferential via normal cones}\label{sect-subdiff}

One possibility to introduce the subdifferential of a function at a point is to view the subgradients as ``slopes'' of continuous affine minorants that coincide with the function at that point. Another perhaps more geometric approach is to use the normal cone to the epigraph of the function. As pointed out previously, we consider the second way more suitable for the nonlinear setting due to the important role of the metric projection in this approach. 

\begin{definition}\label{def-subdiff}
Let $f : X \to (-\infty,\infty]$ be convex, and let $x \in \dom f$. We say that $([\gamma],r) \in T_x X$ is a {\it subgradient} of $f$ at $x$ if $([\sigma],1) \in N_{\epi f}(x,f(x))$, where $\sigma : [0,R_x] \to X \times \R$ is defined by 
\[\sigma(t) = \left(\gamma\left(rt/\sqrt{r^2+1}\right), f(x) - t/\sqrt{r^2+1}\right)\quad \text{for all }t \in [0,R_x].\]

The set of all subgradients of $f$ at $x$ forms the {\it subdifferential} of $f$ at $x$, denoted by $\partial f(x)$.
\end{definition}

Regarding the previous definition, we observe the following.

\begin{remark}\label{rmk-def-subdiff} 
\begin{enumerate}
\item[(i)] $\partial f(x)$ is well-defined since, by Lemma \ref{lemma-add-scalar-prod}(ii), we have that $\widetilde{\sigma} \in [\sigma]$ if and only if there exists $\widetilde{\gamma} \in [\gamma]$ such that 
\[\widetilde{\sigma}(t) = \left(\widetilde{\gamma}\left(rt/\sqrt{r^2+1}\right), f(x) - t/\sqrt{r^2+1}\right)\quad \text{for all }t \in [0,R_x].\]
\item[(ii)] If $X$ is a $\CAT(0)$ space with the geodesic extension property, then one can assume that $\gamma$ is actually a geodesic ray, and so $\sigma$ can be defined on any interval $[0,l] \subseteq [0,\infty)$.
\end{enumerate}
\end{remark}

A key point in understanding the behavior of the subdifferential of convex functions is the fact that such functions that are continuous at a point are locally Lipschitz there.

\begin{lemma}\label{lemma-bd-Lip}
Let $(X,d)$ be a  uniquely geodesic space, let $f : X \to (-\infty, \infty]$ be convex, and let $x \in \dom f$. Suppose that $X$ has the geodesic extension property around $x$. Then the following are equivalent:
\begin{enumerate}
\item[(i)] $f$ is continuous at $x$.
\item[(ii)] $f$ is bounded on a neighborhood of $x$.
\item[(iii)] $f$ is Lipschitz on a neighborhood of $x$.
\end{enumerate}
\end{lemma}
\begin{proof}
We only prove that (ii) implies (iii) since the other implications are obvious. Suppose there exists $r > 0$ such that for all $y \in B(x,r)$, $|f(y)| \le M$ for some $M \ge 0$. We show that $f$ is Lipschitz on $B(x, R)$, where $R = \min\{r/2,R_x/4\}$. To this end, let $y,z \in B(x, R)$ be distinct. Then $d(y,z) < R_x/2$ and we can choose $w \in X$ with $z = (1-t)y + tw$, where $t=d(y,z)/(d(y,z) + R)$. Note that $d(w,z) = R$. Also, $d(x,w) \le d(x,z) + d(z,w) < r$, so $w \in B(x,r)$. Because $f$ is convex, $f(z) \le (1-t)f(y) + tf(w)$, and hence
\[f(z) - f(y) \le t(f(w)-f(y)) \le 2tM = 2M d(y,z) \frac{1}{d(y,z) + R} \le \frac{2M}{R}d(y,z).\]
By swapping the roles of $x$ and $y$ in the above argument, we finally get $|f(y)-f(z)| \le (2M/R)d(y,z)$.
\end{proof}

We characterize below the subdifferential of a convex function by means of a variational inequality. This fact will be essential to prove the first order condition for  minimizers of convex functions. We return to the setting of a $\CAT(\kappa)$ space (with the additional conditions assumed in Section \ref{sect-tangent-sp}).

\begin{proposition} \label{prop-subdiff-convex}
Let $f : X \to (-\infty,\infty]$ be convex, let $x \in \dom f$, and let $([\gamma],r) \in T_x X$. For $y \in X$, denote by $\gamma^y$ the geodesic from $x$ to $y$. Then $([\gamma],r) \in \partial f(x)$ if and only if 
\begin{equation}\label{prop-subdiff-convex-eq1}
\langle ([\gamma],r), ([\gamma^y],d(x,y))\rangle  + f(x) \le f(y) \quad\text{for all } y \in X.
\end{equation}
\end{proposition}
\begin{proof}
Let $\sigma : [0,R_x] \to X \times \R$ be defined by 
\[\sigma(t) = \left(\gamma\left(rt/\sqrt{r^2+1}\right), f(x) - t/\sqrt{r^2+1}\right)\quad \text{for all }t \in [0,R_x].\]
Taking $u = rR_x/\sqrt{r^2+1}$ and $v=R_x/\sqrt{r^2+1}$, we have $u^2 + v^2 = R_x^2$ and
\[\sigma(t) = \left(\gamma\left(ut/R_x\right), f(x) - vt/R_x\right)\quad \text{for all }t \in [0,R_x].\]
By Remark \ref{rmk-normal-cone}(v), $([\sigma],1) \in N_{\epi f}(x,f(x))$ if and only if for all $y \in \dom f$ and $a \ge 0$,
\begin{equation}\label{prop-subdiff-convex-eq2}
\langle ([\sigma],R_x), ([\sigma^{y,a}],\delta) \rangle \le 0,
\end{equation} 
where $\delta = d_2((x,f(x)),(y,f(y)+a))$ and $\sigma^{y,a} : [0,\delta] \to \epi f$ is the geodesic from $(x,f(x))$ to $(y,f(y)+a)$.

Denoting $\alpha : [0,v] \to \R$, $\alpha(t) = f(x) - t$, $\gamma^y : [0,d(x,y)] \to X$ the geodesic from $x$ to $y$, and $\beta : [0,|f(y) + a - f(x)|] \to \R$ the geodesic from $f(x)$ to $f(y) + a$, using Lemma \ref{lemma-add-scalar-prod}, inequality \eqref{prop-subdiff-convex-eq2} amounts to
\[\langle ([\gamma],u), ([\gamma^{y}],d(x,y)) \rangle + \langle ([\alpha],v), ([\beta],|f(y) + a - f(x)|) \rangle \le 0,\]
or, equivalently,
\begin{equation}\label{prop-subdiff-convex-eq3}
\langle ([\gamma],r), ([\gamma^{y}],d(x,y))\rangle + f(x)  \le f(y) + a.
\end{equation}

Now, if $([\gamma],r) \in \partial f(x)$, then taking $a=0$ in \eqref{prop-subdiff-convex-eq3}, we get \eqref{prop-subdiff-convex-eq1}. Conversely, if \eqref{prop-subdiff-convex-eq1} holds, then \eqref{prop-subdiff-convex-eq3} holds for all $a \ge 0$, and hence $([\gamma],r) \in \partial f(x)$.
\end{proof}

\begin{corollary}
Let $f : X \to (-\infty,\infty]$ be convex and proper. Then 
\[\argmin f = \{x \in X \mid o_x \in \partial f(x)\}.\]
\end{corollary}
\begin{proof}
Let $x \in X$. Then $x \in \argmin f$ if and only if $\langle o_x, ([\gamma^y],d(x,y))\rangle  + f(x) \le f(y)$ for all $y \in X$. By Proposition \ref{prop-subdiff-convex}, this is equivalent to $o_x \in \partial f(x)$.
\end{proof}

\begin{remark}\label{rmk-notion-GigNob}
Recently, with the purpose of approaching gradient flows in $\CAT(\kappa)$ spaces from a differential viewpoint, the object minus-subdifferential was introduced in \cite[Section 3]{GigNob21}. In our context and terminology, this definition can be stated as follows:  given $f : X \to (-\infty,\infty]$ convex and $x \in \dom f$, we say that $([\eta],r) \in T_x X$ belongs to the minus-subdifferential of $f$ at $x$ if 
\begin{equation}\label{rmk-notion-GigNob-eq1}
-\langle ([\eta],r), ([\gamma^y],d(x,y))\rangle  + f(x) \le f(y) \quad\text{for all } y \in X,
\end{equation}
where $\gamma^y$ is the geodesic from $x$ to $y$.

Note that if $([\gamma],r), ([\eta],r) \in T_x X$ are opposite to each other and $w \in T_x X$, then we have $\langle ([\gamma],r), w\rangle \le -\langle ([\eta],r), w\rangle$, with no equality in general. Thus, if $([\eta],r)$ satisfies $\eqref{rmk-notion-GigNob-eq1}$, then $([\gamma],r)$ satisfies \eqref{prop-subdiff-convex-eq1}, and hence $([\gamma],r) \in \partial f(x)$. 
\end{remark}

A first step to develop a subdifferential calculus is the description of the subdifferential for functions such as the indicator function, the squared distance to a point, or the distance to a set.

\begin{example}\label{ex-indicator-fct}
Let $C \subseteq X$ be nonempty and convex. Consider the indicator function $\delta_C : X \to [0,\infty]$ defined by
\[\delta_C(x) = \begin{cases} 0, & \mbox{if } x \in C,\\ \infty, & \mbox{if } x \notin C.\\ \end{cases}\]
Then $\partial \delta_C(x) = N_C(x)$ for all $x \in C$.
\end{example}
\begin{proof}
Let $x \in C$. Since $\delta_C$ is a convex function, by Proposition \ref{prop-subdiff-convex}, $([\gamma],r) \in \partial \delta_C(x)$ if and only if $\langle ([\gamma],r), ([\gamma^y],d(x,y))\rangle \le \delta_C(y) \text{ for all } y \in X$, which is equivalent to $\langle ([\gamma],r), ([\gamma^y],d(x,y))\rangle \le 0 \text{ for all } y \in C$, that is, $([\gamma],r) \in N_C(x)$ according to Remark \ref{rmk-normal-cone}(v).
\end{proof}

\begin{example} \label{ex-squared-dist}
Suppose additionally that $X$ has diameter smaller than $D_\kappa/2$. Fix $z \in X$. If $x \in X$ with $x \ne z$, denote by $\Gamma_x$ the set of all nonconstant geodesics $\gamma : [0,l] \to X$ issuing at $x$ with the property that $d(z,x) + l = d(z,\gamma(l))$. By the geodesic extension property around $x$, $\Gamma_x \ne \emptyset$.  

Define $f : X \to \R$ by $f(x) = \frac{1}{2} d(x,z)^2$. Then 
\[\partial f(x) = \begin{cases} \{([\gamma],d(z,x)) \in T_x X \mid \gamma \in \Gamma_x\}, & \mbox{if } x \ne z,\\ o_z, & \mbox{if } x =z.\\ \end{cases}\]
\end{example}
\begin{proof}
Since $X$ is $2$-uniformly convex with parameter $K > 0$, the function $f$ is convex.

Let $x \in X$ and $([\gamma],r) \in \partial f(x)$. Then $([\sigma],1) \in N_{\epi f}(x,f(x))$, where $\sigma : [0, R_x] \to X \times \R$: 
\[\sigma(t) = \left(\gamma\left(rt/\sqrt{r^2+1}\right), f(x) - t/\sqrt{r^2+1}\right) \quad \text{for all }t \in [0,R_x].\]
Thus, $P_{\epi f}(\sigma(R_x)) = (x,f(x))$. If $u = R_x/\sqrt{r^2 + 1}$, then $ur < R_x$. Take $y = (1-a)z + az'$, where $a = 1/(u+1)$ and $z' = \gamma(ur)$. Then, for all $t\in (0,1)$,
\[d_2(\sigma(R_x),(x,f(x)))^2 \le d_2(\sigma(R_x),(1-t)(x,f(x)) + t(y,f(y)))^2,\]
from where
\begin{align*}
d(z',x)^2 + u^2 &\le d(z',(1-t)x + t y)^2 + \left(f(x) - u - (1-t)f(x) - t f(y)\right)^2\\
& \le (1-t)d(z',x)^2 + t d(z',y)^2 - Kt(1-t)d(x,y)^2\\
& \quad + u^2 - 2tu(f(x) - f(y)) + t^2(f(x)-f(y))^2.
\end{align*}
Thus,
\[Kt(1-t)d(x,y)^2 \le -t d(z',x)^2 + t(1-a)^2d(z',z)^2 - tu \, d(x,z)^2 + tu \, d(y,z)^2+ t^2(f(x)-f(y))^2.\]
Dividing by $t$ and then letting $t \searrow 0$, we get
\begin{align*}
Kd(x,y)^2 & \le -d(z',x)^2 + \left((1-a)^2 + ua^2\right)d(z',z)^2 - u\, d(x,z)^2\\
& \le \frac{u}{u+1}\left(d(z',x) + d(x,z)\right)^2  -d(z',x)^2 - u\, d(x,z)^2 \\
& = u\left(\frac{1}{u+1}(ur + d(x,z))^2 - ur^2 - d(x,z)^2\right) = - \frac{u^2}{u+1}(r-d(x,z))^2 \le 0.
\end{align*}
From the above inequality, we conclude that $x=y$ and $r = d(x,z)$. If $x=z$, then $r = 0$ and $([\gamma],r) = o_z$. Otherwise, $\gamma \in \Gamma_x$.

Let $x \in X$ with $x \ne z$ and $\gamma \in \Gamma_x$. Taking an extension of $\gamma$ if needed, denote $r = d(z,x)$, $u = R_x/\sqrt{r^2 + 1}$, $z' = \gamma(ur)$, and $a = 1/(u+1)$. Then 
\[x = (1-a)z+az', \quad d(x,z') = ur, \quad \text{and} \quad d(z,z') = (u+1)r.\] 
We prove that $([\gamma],r) \in \partial f(x)$. More precisely, we show that $([\sigma],1) \in N_{\epi f}(x,f(x))$, where $\sigma : [0,R_x] \to X \times \R$:
\[\sigma(t) = \left(\gamma\left(rt/\sqrt{r^2+1}\right),f(x)-t/\sqrt{r^2+1}\right) \quad \text{for all }t \in [0,R_x].\]
Let $(y,\lambda) \in \epi f$. Then
\begin{align*}
d_2(\sigma(R_x),(y,\lambda))^2 & = d(z',y)^2 + (f(x) - u - \lambda)^2\\
&  = u^2 + d(z',y)^2 - u\, d(z,x)^2 + 2u\lambda + (f(x) - \lambda)^2\\
& \ge u^2 +  d(z',y)^2 - ur^2 +u \, d(z,y)^2.
\end{align*}
Since $d(z',y) \ge |d(z',z) - d(z,y)| = |(u+1)r - d(z,y)|$, it follows that
\[d(z',y)^2 - ur^2 +u \, d(z,y)^2 \ge u^2r^2 + (u+1)(r-d(z,y))^2 \ge u^2r^2,\]
from where 
\[d_2(\sigma(R_x),(y,\lambda))^2 \ge u^2 + u^2r^2 = d(z',x)^2 + u^2 = d_2(\sigma(R_x),(x,f(x)))^2.\]
Hence, $P_{\epi f}(\sigma(R_x)) = (x,f(x))$, so $([\sigma],1) \in N_{\epi f}(x,f(x))$.

Finally, we prove that $o_z \in \partial f(z)$, which amounts to showing that $([\sigma],1) \in N_{\epi f}(z,0)$, where $\sigma : [0,R_x] \to X \times \R$, $\sigma(t) = (z,-t)$ for all $t \in [0,R_x]$. Since $\epi f \subseteq X \times \R_+$, this follows by Example \ref{ex-normal-cone}(iv) and Remark \ref{rmk-normal-cone}(iii).
\end{proof}

In the above example, the condition that the diameter of $X$ is smaller than $D_\kappa/2$ is used to apply $2$-uniform convexity. Thus, when $X$ is a $\CAT(0)$ space, its diameter can be unbounded. Moreover, in $\CAT(0)$ spaces, the distance function to a convex set is convex and we also have the following example.

\begin{example}\label{ex-dist-set}
Consider $(X,d)$ a complete, locally compact $\CAT(0)$ space with the geodesic extension property. Let $C \subseteq X$ be nonempty, convex, and closed. If $x \in X \setminus C$, denote by $\Gamma_x$ the set of all geodesics $\gamma : [0,1] \to X$ issuing at $x$ with the property that $d(P_C(x),x) + 1 = d(P_C(x),\gamma(1))$. By the geodesic extension property, $\Gamma_x \ne \emptyset$. 

Define $f : X \to \R$ by $f(x) = \dist(x,C)$. Then 
\[\partial f(x) = \begin{cases} \{([\gamma],1) \in T_x X\mid \gamma \in \Gamma_x\}, & \mbox{if } x \notin C,\\ \{([\gamma],r) \in N_C(x) \mid  r \in [0,1]\}, & \mbox{if } x \in \bd(C), \\ o_x, & \mbox{if } x \in \inte(C).\\ \end{cases}\]
\end{example}

\begin{proof}
Let $x \in X \setminus C$. Denote $z = P_C(x)$ and 
\[a = \frac{d(x,z)}{d(x,z) + 1} \in (0,1).\] 
Clearly, $d(x,z) = a/(1-a)$.

If $([\gamma],r) \in \partial f(x)$, then $([\sigma],1) \in N_{\epi f}(x,f(x))$, where $\sigma : [0, \sqrt{r^2+1}] \to X \times \R$: 
\[\sigma(t) = \left(\gamma\left(rt/\sqrt{r^2+1}\right), f(x) - t/\sqrt{r^2+1}\right) \quad \text{for all }t \in [0,\sqrt{r^2+1}].\]
Thus, $P_{\epi f}(\sigma(\sqrt{r^2+1})) = (x,f(x))$. Take $y = (1-a)z + a\gamma(r)$. Then, for all $t \in (0,1)$, 
\[d_2(\sigma(\sqrt{r^2+1}),(x,f(x)))^2 \le d_2(\sigma(\sqrt{r^2+1}),(1-t)(x,f(x)) + t(y,f(y)))^2.\]
As in the proof of Example \ref{ex-squared-dist}, we get that
\begin{align*}
d(x,y)^2 & \le (1-a)^2d(\gamma(r),z)^2 - d(\gamma(r),x)^2 - 2(f(x)-f(y))\\
& \le (1-a)^2d(\gamma(r),z)^2 - d(\gamma(r),x)^2 - 2d(x,z) + 2d(y,z) \\
& =  (1-a)^2d(\gamma(r),z)^2 - d(\gamma(r),x)^2- 2d(x,z) + 2ad(\gamma(r),z)\\
& \le  (1-a)^2(d(\gamma(r),x) + d(x,z))^2 - d(\gamma(r),x)^2- 2d(x,z) + 2a (d(\gamma(r),x) + d(x,z))\\
& = (1-a)^2\left(r + \frac{a}{1-a}\right)^2 - r^2- 2\frac{a}{1-a} + 2a\left(r + \frac{a}{1-a}\right) \\ 
& = -a(2-a)(r - 1)^2 \le 0,
\end{align*}
from where $x=y$ and $r=1$. This shows that $\gamma \in \Gamma_x$.

Now let $\gamma \in \Gamma_x$. We prove next that $([\gamma],1) \in \partial f(x)$. More precisely, we show that $([\sigma],1) \in N_{\epi f}(x,f(x))$, where $\sigma : [0,\sqrt{2}] \to X \times \R$:
\[\sigma(t) = \left(\gamma\left(t/\sqrt{2}\right),f(x)-t/\sqrt{2}\right)  \quad \text{for all }t \in [0,\sqrt{2}].\]
To this end, we prove that for all $(y,\lambda) \in \epi f$,
\[2 = d_2(\sigma(\sqrt{2}),(x,f(x)))^2\le d_2(\sigma(\sqrt{2}),(y,\lambda))^2.\]
Let $(y,\lambda) \in \epi f$. Denote $D = d_2(\sigma(\sqrt{2}),(y,\lambda))^2$, $z' = \gamma(1)$, $w = P_C(y)$, and $v = (1-a)w + az'$. Then $x = (1-a)z + az'$, $d(z,z') = 1/(1-a)$, $\lambda \ge d(y,w)$, and
\begin{align*}
d(y,v)^2 &\le (1-a)d(y,w)^2 + ad(y,z')^2 - a(1-a)d(w,z')^2\\
& \le (1-a)d(y,w)^2 + ad(y,z')^2 - a(1-a)d(z,z')^2\\ 
&= (1-a)d(y,w)^2 + ad(y,z')^2 - \frac{a}{1-a}.
\end{align*}
Therefore, 
\begin{align*}
D & =  d(z',y)^2 + (d(x,z)-1-\lambda)^2 =  d(z',y)^2 + \left(\frac{1-2a}{1-a} + \lambda\right)^2\\
& \ge \frac{1}{a} d(y,v)^2 - \frac{1-a}{a}d(y,w)^2 + \frac{1}{1-a} + \left(\frac{1-2a}{1-a}\right)^2 + 2\frac{1-2a}{1-a}\lambda + \lambda^2.
\end{align*}

Case I: $a \in (0,1/2)$. Then $1-2a > 0$ and we have
\begin{align*}
D & \ge \frac{1}{a} d(y,v)^2 - \frac{1-a}{a}d(y,w)^2 + \frac{1}{1-a} + \left(\frac{1-2a}{1-a}\right)^2 + 2\frac{1-2a}{1-a}d(y,w) + d(y,w)^2\\
& =  \frac{1}{a}\left(d(y,v)^2 - \left(d(y,w)-d(x,z)\right)^2\right) +2\left(d(y,w)-d(x,z)\right)^2 + 2.
\end{align*}
We show that $|d(y,w) - d(x,z)| \le d(y,v)$. Taking $p = (1-a)w + az \in C$, we have
\begin{align*}
d(y,w) - d(x,z) & \le d(y,p) - d(x,z) \le d(y,v) + d(v,p) - d(x,z)\\
& \le d(y,v) + ad(z',z) - d(x,z) = d(y,v).
\end{align*}
Similarly, taking $q = (1-a)z + aw \in C$, we have
\[d(x,z) - d(y,w) \le d(x,q) - d(y,w) \le ad(z',w) - d(y,w) = d(v,w) - d(y,w) \le d(v,y).\]
Thus, $D \ge 2$.

Case II: $a \in [1/2,1)$. Then $1-2a \le 0$ and we have
\begin{align*}
D & \ge \frac{1}{a} d(y,v)^2 - \frac{1-a}{a}\lambda^2 + \frac{1}{1-a} + \left(\frac{1-2a}{1-a}\right)^2 + 2\frac{1-2a}{1-a}\lambda+ \lambda^2\\
& \ge  \frac{2a-1}{a}\lambda^2 - 2\frac{2a-1}{1-a}\lambda + \frac{1}{1-a} + \left(\frac{1-2a}{1-a}\right)^2 = \frac{2a-1}{a}\left(\lambda - \frac{a}{1-a}\right)^2 + 2 \ge 2.
\end{align*}

Consequently, in both cases, $D \ge 2$, which finishes the proof that $([\gamma],1) \in \partial f(x)$.

Now let $x \in \bd(C)$. Note that $f(x) = 0$. 

Suppose that $([\gamma],r) \in \partial f(x)$ with $r>0$. Then $([\sigma],1) \in N_{\epi f}(x,0)$, where $\sigma : [0, \sqrt{r^2+1}] \to X \times \R$:
\[\sigma(t) = \left(\gamma\left(rt/\sqrt{r^2+1}\right), - t/\sqrt{r^2+1}\right) \quad \text{for all }t \in [0,\sqrt{r^2 + 1}].\]
Since $P_{\epi f}(\sigma(\sqrt{r^2+1})) = (x,0)$, it follows that for all $(y,\lambda) \in \epi f$,
\[d_2(\sigma(\sqrt{r^2+1}),(x,0))^2 \le d_2(\sigma(\sqrt{r^2+1}),(y,\lambda))^2,\]
from where 
\begin{equation} \label{ex-distance-set-eq1}
d(\gamma(r),x)^2 + 1 \le d(\gamma(r),y)^2 + (\lambda+1)^2.
\end{equation}
Thus, for all $y \in C$, $d(\gamma(r),x) \le d(\gamma(r),y)$, so $x = P_C(\gamma(r))$, which yields $([\gamma],1) \in N_C(x)$.

Let $y = (1-t)x + t\gamma(r)$, where $t \in (0,1)$. Then $f(y) = tr$, $d(\gamma(r),y) = (1-t)r$, and, applying \eqref{ex-distance-set-eq1} for $(y,f(y)) \in \epi f$ we obtain
\[r^2 + 1 \le (1-t)^2r^2 + (tr+1)^2.\] 
Thus, $r^2(1-t) - r \le 0$. Letting $t \searrow 0$, this yields $r \le 1$.

Now take $([\gamma],r) \in N_C(x)$ with $r \in [0,1]$. Then $P_C(\gamma(r))= x$. We show that $([\gamma],r) \in \partial f(x)$, i.e, $([\sigma],1) \in N_{\epi f}(x,0)$, where $\sigma : [0, \sqrt{r^2+1}] \to X \times \R$: 
\[\sigma(t) = \left(\gamma\left(rt/\sqrt{r^2+1}\right), - t/\sqrt{r^2+1}\right) \quad \text{for all }t \in [0,\sqrt{r^2 + 1}].\]
Let $(y, \lambda) \in \epi f$. If $d(\gamma(r),y) \ge 1$, then $d(\gamma(r),x) = r \le 1 \le d(\gamma(r),y)$. Otherwise, if $d(\gamma(r),y) < 1$, denoting $z = P_C(y)$, we have
\begin{align*}
d(\gamma(r),x)^2 & \le d(\gamma(r),z)^2 \le \left(d(\gamma(r),y) + d(y,z)\right)^2 \le d(\gamma(r),y)^2 + 2d(y,z) + d(y,z)^2\\
& = d(\gamma(r),y)^2 + (f(y)+1)^2 - 1 \le d(\gamma(r),y)^2 + (\lambda + 1)^2 - 1.
\end{align*}
Therefore, $d_2(\sigma(\sqrt{r^2+1}),(x,0)) \le d_2(\sigma(\sqrt{r^2+1}),(y,\lambda))$, so $P_{\epi f}(\sigma(\sqrt{r^2+1})) = (x,0)$ and $([\sigma],1) \in N_{\epi f}(x,0)$.

Finally, consider $x \in \inte(C)$. Again, $f(x) = 0$. If $([\gamma],r) \in \partial f(x)$, as before, one has that $x = P_C(\gamma(r))$. Since $x \in \inte(C)$, this implies that $r=0$. Conversely, to show that $o_x \in \partial f(x)$, one argues as in Example \ref{ex-squared-dist}.
\end{proof}

\begin{remark}
In Example \ref{ex-dist-set}, if $C = \{z\}$ for some $z \in X$, then $f : X \to \R$, $f(x) = d(x,z)$, and, for $x \ne z$, $\Gamma_x$ is the set of all geodesics $\gamma : [0,1] \to X$ issuing at $x$ with the property that $d(z,x) + 1 = d(z,\gamma(1))$. We then have
\[\partial f(x) = \begin{cases} \{([\gamma],1) \in T_x X\mid \gamma \in \Gamma_x\}, & \mbox{if } x \ne z,\\ \{([\gamma],r) \in T_x X \mid  r \in [0,1]\}, & \mbox{if } x = z.\\ \end{cases}\]

If $C = \overline{B}(z,R)$ for some $z \in X$ and $R > 0$, then $f : X \to \R$, $f(x) = \dist(x,\overline{B}(z,R))$, and, for $x \notin B(z,R)$, $\Gamma_x$ is the set of all geodesics $\gamma : [0,1] \to X$ issuing at $x$ with the property that $d(z,x) + 1 = d(z,\gamma(1))$. We then have
\[\partial f(x) = \begin{cases} \{([\gamma],1) \in T_x X\mid \gamma \in \Gamma_x\}, & \mbox{if } d(x,z) > R,\\ \{([\gamma],r) \in T_x X \mid  r \in [0,1], \gamma \in \Gamma_x\}, & \mbox{if } d(x,z) = R, \\ o_x, & \mbox{if } d(x,z) < R.\\ \end{cases}\]
\end{remark}

We show now that, as in the linear setting, the subdifferential of a continuous convex function is nonempty.
\begin{theorem} \label{thm-nonempty-subdiff}
Consider $f : X \to (-\infty,\infty]$ a convex function and let $x \in \dom f$. If $f$ is continuous at $x$, then $\partial f(x) \ne \emptyset$.
\end{theorem}
\begin{proof}
Since $f$ is continuous at $x$, there exists $\eps \in (0, D_\kappa/2)$ so that $f$ is continuous on $\overline{B}(x,\eps)$.

Suppose $N_{\epi f}(x,f(x)) = \{o_{(x,f(x))}\}$ and denote $B = \overline{B}((x,f(x)),\eps)$. Using Remark \ref{rmk-normal-cone}(iv), $N_{\epi f \cap B}(x,f(x)) = \{o_{(x,f(x))}\}$.  

Note that the set $\epi f \cap B$ is closed. Indeed, let $(y_n,s_n) \subseteq \epi f \cap B$ be a sequence whose limit is $(y,s)$. Clearly, $(y,s) \in B$. Moreover, for all $n \in \N$, $d(y_n,x) \le d_2((y_n,s_n),(x,f(x)) \le \eps$, so $y \in \overline{B}(x,\eps)$ and $f$ is continuous at $y$. As $y_n \to y$, $s_n \to s$, and $f(y_n) \le s_n$, we get $f(y) \le s$, and hence $(y,s) \in \epi f$. 

Applying Proposition \ref{prop-normal-cone-int}(ii), we conclude that  $(x,f(x)) \in \inte(\epi f \cap B)$, which is a contradiction because $(x,f(x)) \in \bd(\epi f)$. 

Thus, there exists $([\sigma],1) \in N_{\epi f}(x,f(x))$, where $\sigma : [0,R_x] \to X \times \R$ is a nonconstant geodesic. Write $R_x^2 = r^2 + s^2$ and $\sigma : [0,R_x] \to X \times \R$ as $\sigma(t) = \left(\gamma(rt/R_x), \alpha(st/R_x)\right)$ for all $t \in [0,R_x]$, where $\gamma : [0,r] \to X$ and $\alpha : [0,s] \to \R$ are geodesics. Then $\alpha(t) = f(x) + t$ for all $t \in [0,s]$ or $\alpha(t) = f(x) - t$ for all $t \in [0,s]$.

For $y \in \dom f$ and $a \ge 0$, consider $\gamma^y : [0,d(x,y)] \to X$ the geodesic from $x$ to $y$ and $\beta : [0,|f(y) + a - f(x)|] \to \R$ the geodesic from $f(x)$ to $f(y) + a$. As in the proof of Proposition \ref{prop-subdiff-convex}, we have
\[\langle ([\gamma],r), ([\gamma^{y}],d(x,y)) \rangle + \langle ([\alpha],s), ([\beta],|f(y) + a - f(x)|) \rangle \le 0.\]

Suppose first that $s=0$. Then $r > 0$ and $\langle ([\gamma],r), ([\gamma^{y}],d(x,y)) \rangle \le 0$ for all $y \in \dom f$. If $y \in \gamma((0,\min\{\eps,r\}])$, then $y \in \dom f$ and $\angle(\gamma,\gamma^y) = 0$, so $0 < rd(x,y)\cos\angle(\gamma,\gamma^y) = \langle ([\gamma],r), ([\gamma^{y}],d(x,y)) \rangle \le 0$, a contradiction. Thus, $s > 0$. We distinguish the following two cases.

Case I: $\alpha(t) = f(x) + t$ for all $t \in [0,s]$. Then
\[\langle ([\gamma],r), ([\gamma^{y}],d(x,y)) \rangle + s(f(y) + a - f(x)) \le 0,\]
for any $y \in \dom f$ and $a \ge 0$. Taking $y=x$ and $a=1$, we obtain again a contradiction.

Case II: $\alpha(t) = f(x) - t$ for all $t \in [0,s]$. Denoting $u = r/s$, we have $r/R_x = u/\sqrt{u^2 + 1}$, $s/R_x = 1/\sqrt{u^2 + 1}$. Hence,
\[\sigma(t) = \left(\gamma(ut/\sqrt{u^2 + 1}), f(x) - t/\sqrt{u^2+1}\right) \quad \text{for all }t \in [0,R_x]\]
and we conclude that $([\gamma],u) \in \partial f(x)$. 
\end{proof}

\begin{remark}\label{rmk-bd-decay}
If $f : X \to (-\infty, \infty]$ is a convex function and $\partial f(x) \ne \emptyset$ for some $x \in \dom f$ (e.g., if $f$ is continuous at $x$), then $f$ has bounded decay rate in the sense that there exists a constant $c \ge 0$ such that the function $y \mapsto f(y) + c d(x,y)$ is bounded below.

Indeed, if there exists $([\gamma],c) \in \partial f(x)$, applying Proposition \ref{prop-subdiff-convex}, we get that 
\[f(y) \ge f(x) + \langle ([\gamma],c), ([\gamma^y],d(x,y))\rangle \ge f(x) - c d(x,y),\] 
for all $y \in X$.
\end{remark}

In what follows, we suppose additionally that $X$ has diameter smaller than $D_\kappa/2$. We first recall the following result (see, e.g., \cite[Chapter II, Corollary 3.6]{Bri99}).

\begin{proposition}[First variation formula]\label{prop-FVF}
Let $x,y \in X$ with $x \ne y$, $\gamma : [0,l] \to X$ be a geodesic issuing at $x$ and $\eta$ the geodesic from $x$ to $y$. Then
\[\lim_{t \searrow 0}\frac{d(\gamma(tl),y) - d(x,y)}{t} = - \frac{1}{d(x,y)}\langle ([\gamma],l),([\eta],d(x,y))\rangle.\]
\end{proposition}

\begin{remark}\label{rmk-scalar-prod-squared-dist}
Let $x,y \in X$ with $x \ne y$, $\gamma : [0,l] \to X$ be a geodesic issuing at $x$ and $\eta$ be the geodesic from $x$ to $y$. Then
\[\langle ([\gamma],l), ([\eta],d(x,y))  \rangle \ge \frac{1}{2}\left(d(x,y)^2 - d(\gamma(l),y)^2\right).\]

Indeed, since the function $d(y,\cdot)^2$ is convex, we have that for all $t \in (0,1)$,
\[d(y,\gamma(tl))^2 \le (1-t)d(y,x)^2 + t d(y,\gamma(l))^2,\]
from where
\[ \left(d(y,\gamma(tl)) - d(y,x)\right)\left(d(y,\gamma(tl)) + d(y,x)\right)= d(y,\gamma(tl))^2 - d(y,x)^2 \le t(d(y,\gamma(l))^2 - d(y,x)^2).\]
One obtains the conclusion after dividing by $t$, letting $t \searrow 0$, and applying Proposition \ref{prop-FVF}.
\end{remark}

Consider $f : X \to (-\infty, \infty]$ a proper and convex function that has complete sublevel sets. 

\begin{remark}
Note first that $f$ is lower semicontinuous and that all its sublevel sets are convex. Moreover, by the Hopf-Rinow theorem, we deduce that all sublevel sets of $f$ are also compact. Thus, $f$ attains its minimum and hence is bounded below.

Indeed, denote $a = \inf f(X)$ and let $(x_n) \subseteq X$ such that $f(x_n) \to a$. Take $y \in \dom f$. If $f(y) = a$, then $f$ attains its minimum at $y$. Otherwise, $a < f(y)$, so for $n $ sufficiently large, $x_n$ belongs to the compact set $\{x \in X \mid f(x) \le f(y)\}$. Thus, there exists a subsequence $(x_{n_k})$ of $(x_n)$ that converges to some $x \in X$. By the lower semicontinuity of $f$, we obtain that $f(x) = a$.
\end{remark}

The {\it resolvent} of $f$ is the mapping $J_\lambda^f : X \to \dom f$ defined for $\lambda > 0$ by
\begin{equation}\label{def-resolvent}
J_\lambda^f(x) := \argmin_{y \in X}\left(f(y) + \frac{1}{2\lambda}d(x,y)^2\right) \quad \text{for all } x \in X.
\end{equation}
Observe that there exists a unique point attaining the minimum in \eqref{def-resolvent}. Indeed, existence follows because the sublevel sets of the function $y \in X \mapsto f(y) + \frac{1}{2\lambda}d(x,y)^2$ are closed subsets of sublevel sets of the function $f$ and hence are compact. Uniqueness follows by convexity of $f$ and by $2$-uniform convexity of $X$.

In the following proposition, we establish a connection between the resolvent of $f$ and its subdifferential.
\begin{proposition} \label{prop-resolv-subdiff}
Let $\lambda > 0$, and let $z \in \dom f$. For $x \in X$, denote by $\gamma^x$ the geodesic from $z$ to $x$. Given $x \in X$, we have that $z = J_\lambda^f(x)$ if and only if $([\gamma^x],d(x,z)/\lambda) \in \partial f(z)$.
\end{proposition}
\begin{proof}
Note that, by Proposition \ref{prop-subdiff-convex}, $([\gamma^x],d(x,z)/\lambda) \in \partial f(z)$ if and only if $([\gamma^x],d(z,x)) \in \partial (\lambda f)(z)$.

Suppose first that $z = J_\lambda^f(x)$, and denote $r = d(z,x)$. We prove that $([\sigma],1) \in N_{\epi (\lambda f)}(z,\lambda f(z))$, where $\sigma : [0,\sqrt{r^2 + 1}] \to X \times \R$:
\[\sigma(t) = \left(\gamma^x\left(rt/\sqrt{r^2+1}\right), \lambda f(z)-t/\sqrt{r^2+1}\right).\]
Let $(y,\beta) \in \epi(\lambda f)$. Then $y \in \dom f$, $\beta \ge \lambda f(y)$, and we have 
\begin{align*}
d_2(\sigma(\sqrt{r^2+1}),(y,\beta))^2 & = d_2((x,\lambda f(z)-1),(y,\beta))^2 = d(x,y)^2 + (\lambda f(z) - 1 - \beta)^2\\
&  = 1 + d(x,y)^2 -2\lambda f(z) + 2\beta + (\lambda f(z) - \beta)^2\\
& \ge 1 +  d(x,y)^2 +2\lambda (f(y) - f(z))\\
& \ge 1 + d(z,x)^2 = d_2(\sigma(\sqrt{r^2+1}),(z,\lambda f(z)))^2 \quad \text{by } \eqref{def-resolvent}.
\end{align*}
Hence, $P_{\epi (\lambda f)}(\sigma(\sqrt{r^2+1})) = (z,\lambda f(z))$, so $([\sigma],1) \in N_{\epi (\lambda f)}(z,\lambda f(z))$. This shows that $([\gamma^x],d(z,x)) \in \partial (\lambda f)(z)$.

Suppose now that $([\gamma^x],d(z,x)) \in \partial (\lambda f)(z)$ and let $y \in X$. By Proposition \ref{prop-subdiff-convex},
\[\langle ([\gamma^x],d(z,x)), ([\gamma^y],d(z,y)) \rangle \le \lambda(f(y) - f(z)).\]
Applying Remark \ref{rmk-scalar-prod-squared-dist} (with $\gamma = \gamma^y$ and $\eta = \gamma^x$), we get
\[\frac{1}{2}\left(d(z,x)^2 - d(y,x)^2\right) \le \langle ([\gamma^y],d(z,y)), ([\gamma^x],d(z,x)) \rangle.\]
The above two inequalities yield that
\[f(z) + \frac{1}{2\lambda}d(x,z)^2 \le f(y) + \frac{1}{2\lambda}d(x,y)^2.\]
We conclude that $z = J_\lambda^f (x)$.
\end{proof}

\begin{example}
Let $C \subseteq X$ be nonempty, convex, and complete, and let $z \in C$ and $\lambda > 0$. Then $\partial \delta_C(z) = N_C(z)$ and $J_\lambda^{\delta_C} = P_C$.
In this case, given $x \in X$, Proposition \ref{prop-resolv-subdiff} reduces to $z = P_C(x)$ if and only if $([\gamma^x],d(x,z)/\lambda) \in N_C(z)$.
\end{example}

We give next a necessary condition for a point to be a minimizer of the sum of two convex functions starting with the following two lemmas.

\begin{lemma} \label{lemma-sum-rule}
Let $f : X \to (-\infty,\infty]$ be convex. Given $x \in \dom f$ and $\lambda \ge 0$, define the function $\varphi : X \to (-\infty, \infty]$ by $\varphi(y) = f(y) + \lambda d(y,x)^2$. Then $\partial \varphi(x) = \partial f(x)$. 
\end{lemma}
\begin{proof}
For $y \in X$, denote by $\gamma^y$ the geodesic from $x$ to $y$. 
Let $([\gamma],r) \in \partial f(x)$. We apply Proposition \ref{prop-subdiff-convex} and have that for all $y \in X$,
\[\langle ([\gamma],r), ([\gamma^y],d(x,y))\rangle  + \varphi(x)  = \langle ([\gamma],r), ([\gamma^y],d(x,y))\rangle  + f(x) \le f(y) \le \varphi(y).\]
Therefore, $([\gamma],r) \in \partial \varphi(x)$. 

Conversely, suppose that $([\gamma],r) \in \partial \varphi(x)$. Take $y \in X$ with $y \ne x$ and $z = (1-t)x + ty$ for some $t \in (0,1)$. Then
\[\langle ([\gamma],r), ([\gamma^y],d(x,z))\rangle  + f(x)  = \langle ([\gamma],r), ([\gamma^z],d(x,z))\rangle  + \varphi(x) \le\varphi(z) = f(z) + \lambda d(z,x)^2,\]
from where
\[t\langle ([\gamma],r), ([\gamma^y],d(x,y))\rangle  + f(x)  \le (1-t)f(x) + tf(y) + \lambda t^2d(x,y)^2.\]
After rearranging, dividing by $t$, and then letting $t \searrow 0$, we obtain that 
\[\langle ([\gamma],r), ([\gamma^y],d(x,y))\rangle  + f(x) \le f(y).\]
Since the above inequality clearly also holds for $y=x$, we deduce that $([\gamma],r) \in \partial f(x)$.
\end{proof}

\begin{lemma} \label{lemma-Lip-bd-subgr}
Let $f : X \to (-\infty,\infty]$ be convex, let $x \in \dom f$, and let $R \in (0,R_x)$. Suppose that $f$ is Lipschitz on $B(x,R)$ with Lipschitz constant $L$. If $y \in B(x,R/2)$ and $([\gamma],r) \in \partial f(y)$, then $r \le L$.
\end{lemma}
\begin{proof}
Suppose $r > 0$, and let $z = \gamma(R/2)$. Then $z \in B(x,R)$ and we have
\[r R/2 = \langle ([\gamma],r), ([\gamma],R/2)\rangle \le f(z) - f(y) \le L d(z,y) = L R/2,\]
from where $r \le L$.
\end{proof}

\begin{theorem}\label{thm-sum-rule}
Let $f,g : X \to (-\infty, \infty]$ be two convex functions that have complete sublevel sets.  If $x \in \dom f \cap \dom g$ is a minimum point of $f+g$ and $f$ is continuous at $x$, then there exist $([\gamma],r) \in \partial f(x)$ and $([\eta],s) \in \partial g(x)$ such that $([\gamma],r)$ and $([\eta],s)$ are opposite to each other.
\end{theorem}

\begin{proof}
Given $n \in \N$, define $h_n : X \times X \to (-\infty,\infty]$ by 
\[h_n(y,z) = f(y) + g(z) + d(y,x)^2 + \frac{n}{2}d(y,z)^2.\]
Denote $\alpha = \inf f(X) > -\infty$ and $\beta = \inf g(X) > -\infty$. For any $a \in \R$,
\[\{(y,z) \in X \times X \mid h_n(y,z) \le a\} \subseteq \{y \in X \mid f(y) \le a - \beta\} \times \{z \in X \mid g(z) \le a - \alpha\}.\]
Since both $f$ and $g$ have compact sublevel sets, we deduce that $h_n$ has compact sublevel sets and so attains its minimum at some point $(y_n,z_n) \in X \times X$. Observe that for all $n \in \N$, 
\begin{equation}\label{thm-sum-rule-eq1}
f(y_n) + g(z_n) + d(y_n,x)^2 + \frac{n}{2}d(y_n,z_n)^2 = h_n(y_n,z_n) \le h_n(x,x) = f(x) + g(x),
\end{equation}
from where
\[y_n \in \{y \in X \mid f(y) \le f(x) + g(x) - \beta\} \quad \text{and} \quad z_n \in \{z \in X \mid g(z) \le f(x) + g(x) - \alpha\}.\]
We can therefore suppose, after taking subsequences, that $(y_n)$ and $(z_n)$ converge to some $y^* \in X$ and $z^* \in X$, respectively. By \eqref{thm-sum-rule-eq1},
\[\frac{n}{2}d(y_n,z_n)^2 \le f(x) + g(x) - \alpha - \beta\]
for all $n \in \N$ and we conclude that $d(y_n,z_n) \to 0$. This yields that $y^* = z^*$. Moreover, using again \eqref{thm-sum-rule-eq1} and the lower semicontinuity of the functions $f$ and $g$, we obtain 
\[f(y^*) + g(y^*) + d(y^*,x)^2 \le f(x) + g(x) \le f(y^*) + g(y^*),\] 
and hence $y^*=x$.

Because $f$ is continuous at $x$, we have that $f(y_n) \to f(x)$. Taking limit superior as $n \to \infty$ in \eqref{thm-sum-rule-eq1} gives $\limsup_{n \to \infty} g(z_n) \le g(x)$. Combining this with the lower semicontinuity of $g$, we get that $g(z_n) \to g(x)$.

For all $n \in \N$, denote by $\gamma_n$ and $\eta_n$ the geodesics from $y_n$ to $z_n$ and from $z_n$ to $y_n$, respectively.

Define $\varphi : X \to (-\infty, \infty]$ by $\varphi(y) = f(y) + d(y,x)^2$. Note that $\varphi$ is convex and continuous at $x$ and hence Lipschitz on a neighborhood of $x$ by Lemma \ref{lemma-bd-Lip}. For all $n \in \N$ and all $y \in X$,
\[\varphi(y_n) + g(z_n) +  \frac{n}{2}d(y_n,z_n)^2 = h_n(y_n,z_n) \le h_n(y,z_n) = \varphi(y) + g(z_n) +  \frac{n}{2}d(y,z_n)^2,\]
so $y_n = J_{1/n}^\varphi(z_n)$. Similarly, $z_n = J_{1/n}^g(y_n)$. Proposition \ref{prop-resolv-subdiff} yields $([\gamma_n],nd(y_n,z_n)) \in \partial \varphi(y_n)$ and $([\eta_n],nd(y_n,z_n)) \in \partial g(z_n)$.

If $y_n = z_n$ for some $n \in \N$, we obtain that $y_n = x$. Thus, $o_x \in \partial \varphi(x)$ and $o_x \in \partial g(x)$. By Lemma \ref{lemma-sum-rule}, $o_x \in \partial f(x)$ and the conclusion follows.

Assume next that $y_n \ne z_n$ for all $n \in \N$. For $n \in \N$ sufficiently large, extend $\gamma_n$ beyond $z_n$ to a geodesic $\overline{\gamma}_n$ of length $R_x$. Likewise, extend $\eta_n$ beyond $y_n$ to a geodesic $\overline{\eta}_n$ of length $R_x$. By the betweenness property, $\Img \overline{\gamma}_n \cup \Img \overline{\eta}_n$ is a geodesic segment.

Denote $r_n = n d(y_n,z_n)$ for $n \in \N$. As $([\gamma_n],r_n) \in \partial \varphi(y_n)$, using Lemma \ref{lemma-Lip-bd-subgr}, we get that $(r_n)$ is bounded, so it has a convergent subsequence. By local compactness, for $l > 0$ small enough, $(\overline{\gamma}_n(l))$ and $(\overline{\eta}_n(l))$ have convergent subsequences as well. We can assume that $\overline{\gamma}_n(l) \to u \in X$, $\overline{\eta}_n(l) \to v \in X$, and $r_n \to r \ge 0$. Since $y_n \to x$ and $z_n \to x$, we deduce that $x = (1/2) u + (1/2) v$. Denote by $\gamma$ and $\eta$ the geodesics from $x$ to $u$ and from $x$ to $v$, respectively. Note that $([\gamma],r)$ and $([\eta],r)$ are opposite to each other.

We show now that $([\gamma],r) \in \partial f(x)$. For $n$ large enough, take $\sigma_n : [0,l] \to X \times \R$:
\[\sigma_n(t) = \left(\overline{\gamma}_n\left(r_nt/\sqrt{r_n^2+1}\right), \varphi(y_n) - t/\sqrt{r_n^2+1}\right).\] 
Note that $\sigma_n(t) \to \sigma(t)$ for all $t \in [0,l]$, where $\sigma : [0,l] \to X \times \R$ is defined by
\[\sigma(t) = \left(\gamma\left(rt/\sqrt{r^2+1}\right), \varphi(x) - t/\sqrt{r^2+1}\right).\] 
Because $([\gamma_n],r_n) \in \partial \varphi(y_n)$, we have that $([\sigma_n],1) \in N_{\epi \varphi}(y_n, \varphi(y_n))$. By Remark \ref{rmk-normal-cone}(vi), we obtain that $([\sigma],1) \in N_{\epi \varphi}(x, \varphi(x))$, so $([\gamma],r) \in \partial \varphi(x)$. Applying Lemma \ref{lemma-sum-rule}, $([\gamma],r) \in \partial f(x)$. Similarly, using the fact that $g(z_n) \to g(x)$, one can show that $([\eta],r) \in \partial g(x)$.
\end{proof}

\begin{remark}
Note that if $f,g : X \to (-\infty, \infty]$ are proper, convex, and with complete sublevel sets, then $f+g$ attains its minimum since its sublevel sets are closed subsets of sublevel sets of one of the functions and hence are compact. 

Observe also that, as pointed out in Section \ref{intro}, in the classical version of Theorem \ref{thm-sum-rule}, $f$ only needs to be continuous at some point in the domain of $g$ which is not necessarily a minimizer of $f+g$.
\end{remark}

\begin{remark} \label{rmk-sum-rule-CAT0}
Theorem \ref{thm-sum-rule} also holds in unbounded $\CAT(0)$ spaces. More precisely, suppose that $X$ is a complete, locally compact $\CAT(0)$ space with the geodesic extension property, and let $f, g : X \to (-\infty, \infty]$ be two proper, convex, and lower semicontinuous functions. First, observe that even though the sublevel sets of $f$ and $g$ are not necessarily compact, their resolvents are well-defined (see, e.g., \cite{Jo95, May98}). We point out below how one could modify the argument from the proof of Theorem \ref{thm-sum-rule} in order to show that if $x \in \dom f \cap \dom g$ is a minimum point of $f+g$, $f$ is continuous at $x$, and $\partial g(\overline{x}) \ne \emptyset$ for some $\overline{x} \in \dom g$, then there exist $([\gamma],r) \in \partial f(x)$ and $([\eta],s) \in \partial g(x)$ such that $([\gamma],r)$ and $([\eta],s)$ are opposite to each other. 

Indeed, note first that, by Remark \ref{rmk-bd-decay}, $f$ and $g$ have bounded decay rate. Thus, there exist the constants $c_1, c_2 \ge 0$ and  $k_1, k_2 \in \R$ such that $f(y) + c_1 d(y,x) \ge k_1$ and $g(y) + c_2 d(y,\overline{x}) \ge k_2$ for all $y \in X$. Take 
\[c = \max\{c_1, c_2\} \quad \text{and} \quad k = \min\{k_1, k_2 - c_2 d(x,\overline{x})\}.\] 
Then, applying the triangle inequality, for all $y \in X$ we have
\[f(y) + cd(y,x) \ge k \quad \text{and} \quad g(y) + cd(y,x) \ge k.\]
Given $n \in \N$, define $h_n : X \times X \to (-\infty,\infty]$ by 
\[h_n(y,z) = f(y) + d(y,x)^2 + g(z) + d(z,x)^2 + \frac{n}{2}d(y,z)^2.\]
Then 
\begin{align*}
h_n(y,z) &\ge k - c d(y,x) + d(y,x)^2+ k - c d(z,x) + d(z,x)^2 + \frac{n}{2}d(y,z)^2\\
&  \ge \left(d(y,x) - \frac{c}{2}\right)^2 + \left(d(z,x) - \frac{c}{2}\right)^2 + 2\left(k - \frac{c^2}{4}\right),
\end{align*}
for all $y, z \in X$. Let $b = 2(k - c^2/4)$. Given $a \in \R$ with $a \ge b$, denote 
\[Y_a = \{y \in X \mid (d(y,x) - c/2)^2 \le a - b\} \quad \text{and} \quad Z_a = \{z \in X \mid (d(z,x) - c/2)^2 \le a - b\}.\]
Then $Y_a$ and $Z_a$ are bounded and closed and hence compact. Since 
\[\{(y,z) \in X \times X \mid h_n(y,z) \le a\} \subseteq Y_a \times Z_a,\]
we conclude that the sublevel sets of $h_n$ are compact, so $h_n$ attains its minimum at some $(y_n,z_n) \in X \times X$. We then get that
$y_n \in Y_{f(x)+g(x)}$ and $z_n \in Z_{f(x)+g(x)}$ for all $n \in \N$. Therefore, we can suppose, after taking subsequences, that $(y_n)$ and $(z_n)$ converge to some $y^* \in X$ and $z^* \in X$. Similarly as before, one can show that $y^* = z^* = x$, and considering the functions  $\varphi, \psi : X \to (-\infty, \infty]$ defined by $\varphi(y) = f(y) + d(y,x)^2$ and $\psi(z) = g(z) + d(z,x)^2$, one writes $h_n(y,z) = \varphi(y) + \psi(z) + \frac{n}{2}d(y,z)^2$ and proceeds as in the proof of Theorem \ref{thm-sum-rule}.

\end{remark}

We finish with a consequence of the above result which, in a finite-dimensional Hilbert space, yields the existence of a separating hyperplane for two nonempty, convex, closed, and disjoint sets, at least one of which is bounded. It also provides a version in $\CAT(0)$ spaces of a recent separation result for convex sets proved in \cite[Theorem 4.5]{SilBerHer22} in the setting of Hadamard manifolds. 
\begin{corollary}\label{cor-separation}
Let $(X,d)$ be a complete, locally compact $\CAT(0)$ space with the geodesic extension property. If $C, D \subseteq X$ are nonempty, convex, closed, and disjoint and at least one of them is bounded, then there exist $x \in C$ and $([\gamma],1), ([\eta],1) \in T_{x}X$ opposite to each other such that
\[\sup\{\langle ([\gamma],1),([\gamma^y],d(x,y)) \rangle \mid  y \in D\} < 0\]
and
\[\sup \{\langle ([\eta],1),([\gamma^y],d(x,y)) \rangle \mid y \in C\}\le 0,\]
where $\gamma^y$ denotes the geodesic from $x$ to $y$.
\end{corollary}
\begin{proof}
Define $f,g : X \to (-\infty, \infty]$ by $f(x) = \dist(x,D)$ and $g(x) = \delta_C(x)$ for all $x \in X$. These functions are proper and convex. Moreover, $f$ is continuous, while $g$ is lower semicontinuous. 

Denote by $\rho$ the distance between the sets $C$ and $D$. For $n \in \N$, choose $x_n \in C$ and $y_n \in D$ such that $d(x_n,y_n) < \rho + 1/n$. Since one of the sets $C$ or $D$ is bounded, one of the sequences $(x_n)$ or $(y_n)$ is bounded, and hence both sequences are bounded. By the Hopf-Rinow theorem, there exists $x \in C$ such that $d(x,P_D(x)) = \rho$.

Since $x$ is a minimum point of $f+g$, we can use Remark \ref{rmk-sum-rule-CAT0} and Examples \ref{ex-indicator-fct} and \ref{ex-dist-set} to obtain $([\gamma],1), ([\eta],1) \in T_{x}X$ opposite to each other such that $([\gamma],1) \in \partial f(x)$ and $([\eta],1) \in N_C(x)$. Applying Proposition \ref{prop-subdiff-convex}, $\langle ([\gamma],1),([\gamma^y],d(x,y)) \rangle \le - f(x)$ for all $ y \in D$, which yields the first desired inequality. The second one follows by using Remark \ref{rmk-normal-cone}(v).
\end{proof}

\section*{Acknowledgements}
This work was supported in part by National Science Foundation Grant DMS-2006990, by DGES Grant PGC2018-098474-B-C21, and by a grant of the Ministry of Research, Innovation and Digitization, CNCS/CCCDI -- UEFISCDI, project number PN-III-P1-1.1-TE-2019-1306, within PNCDI III.

\end{document}